\documentclass[a4wide]{article}

\usepackage{textcomp}
\usepackage[utf8]{inputenc}
\usepackage[T1]{fontenc}

\usepackage{hyperref}
\usepackage[dvipsnames]{xcolor}
\hypersetup{
    colorlinks,
    linkcolor={red!60!black},
    citecolor={green!40!black}
}
\definecolor{reasonablegreen}{rgb}{0,0.5,0}

\usepackage{amssymb,amsmath,amsthm,amsfonts,amssymb,amsopn}
\usepackage{mathtools}
\usepackage{algorithm,algorithmic}
\usepackage{stmaryrd}
\usepackage{booktabs}
\usepackage{color}
\usepackage{nicefrac}
\usepackage{multirow}

\usepackage{graphicx}        
\usepackage{latexsym}
\usepackage{mathrsfs}
\usepackage{psfrag}
\usepackage{verbatim}
\usepackage{fancyhdr}
\usepackage{enumerate}
\usepackage{ulem}
\input{insbox}
\usepackage{url}

\usepackage{fancybox}

\newcommand{\transpose}[1]{{#1}^{\top}}
\newcommand{\val}[1]{\mathsf{val}(#1)}
\newcommand{\B}[1]{\textbf{#1}}
\newcommand{\centroid}[1]{\mathsf{centroid}(#1)}

\newcommand{\gab}[1]{{\color{black}#1}}
\newcommand{\gabjulrev}[1]{{\color{black}#1}}

\newcommand{\leorev}[1]{{\color{black}#1}}

\newtheorem{result}{\ }[section]
\theoremstyle{changebreak}                
\newtheorem{theorem}[result]{Theorem}
\newtheorem{definition}[result]{Definition}
\newtheorem{lemma}[result]{Lemma}
\newtheorem{corollary}[result]{Corollary}
\newtheorem{proposition}[result]{Proposition}

\newtheorem{remark}[result]{Remark}

\begin{document}

\thispagestyle{empty} 

\noindent\fbox{%
    \parbox{\textwidth}{%
        This paper has been published online and is downloadable at: \url{https://ojmo.centre-mersenne.org/articles/10.5802/ojmo.18/}. Please visit the publisher's website.
    }%
}

\begin{center}
 {\Large Cycle-based formulations in Distance Geometry}
\end{center}

\vspace{7mm}

\noindent\textbf{Leo Liberti}\hfill\href{mailto:liberti@lix.polytechnique.fr}{\ttfamily liberti@lix.polytechnique.fr}\\
\emph{\small CNRS LIX, \'Ecole Polytechnique, F-91128 \\Palaiseau, France}\\
\\
\textbf{Gabriele Iommazzo}\hfill\href{mailto:iommazzo@zib.de}{\ttfamily iommazzo@zib.de}\\
\emph{\small Zuse Institute Berlin\\
Berlin, Germany}\\
\\
\textbf{Carlile Lavor}\hfill\href{mailto:clavor@ime.unicamp.br}{\ttfamily clavor@ime.unicamp.br}\\
\emph{\small IMECC, University of Campinas\\Brazil}\\
\\
\textbf{Nelson Maculan}\hfill\href{mailto:maculan@cos.ufrj.br}{\ttfamily maculan@cos.ufrj.br}\\
\emph{\small COPPE, Federal Univ.~Rio de Janeiro (UFRJ)\\Brazil}\\
\\

\vspace{5mm}

\begin{center}
\begin{minipage}{0.85\textwidth}
\begin{center}
 \textbf{Abstract}
\end{center}
{\small The distance geometry problem asks to find a realization of a given simple edge-weighted graph in a Euclidean space of given dimension $K$, where the edges are realized as straight segments of lengths equal (or as close as possible) to the edge weights. The problem is often modelled as a mathematical programming formulation involving decision variables that determine the position of the vertices in the given Euclidean space. Solution algorithms are generally constructed using local or global nonlinear optimization techniques. We present a new modelling technique for this problem where, instead of deciding vertex positions, the formulations decide the length of the segments representing the edges in each cycle in the graph, projected in every dimension. We propose an exact formulation and a relaxation based on a Eulerian cycle. We then compare computational results from protein conformation instances obtained with stochastic global optimization techniques on the new cycle-based formulation and on the existing edge-based formulation. While edge-based formulations take less time to reach termination, cycle-based formulations are generally better on solution quality measures.
}
\end{minipage}
\end{center}

\vspace{0mm}

\section{Introduction}
\label{s:intro}
We consider the fundamental problem in Distance Geometry (DG):
\begin{quote}
\textsf{Distance Geometry Problem} (DGP). Given a positive integer $K$ and a simple undirected graph $G=(V,E)$ with an edge weight function $d:E\to\mathbb{R}_{\ge 0}$, establish whether there exists a {\it realization} $x:V\to\mathbb{R}^K$ of the vertices such that Eq.~\eqref{dgp} below is satisfied:
\begin{equation}
  \forall \{i,j\}\in E\qquad \|x_i-x_j\|=d_{ij},\label{dgp}
\end{equation}
where $x_i\in\mathbb{R}^K$ for each $i\in V$ and $d_{ij}$ is the weight on edge $\{i,j\}\in E$.
\end{quote}
Although the DGP is given above in the canonical decision form, we consider the corresponding search problem, where one has to actually find the realization $x$. The DGP is also known as the {\it graph realization problem} in geometric rigidity \cite{laurent97,berg,eren04}. It belongs to a more general class of metric completion and embedding problems \cite{kopperman,hoffman,schaeffer2}.

In its most general form, the DGP might be parametrized over any norm \cite{oneinfnorm-lncs}. In practice, the $\ell_2$ norm is the most usual choice \cite{dgp-sirev}, and will also be employed in this paper. The DGP with the $\ell_2$ norm is sometimes called the {\sc Euclidean DGP} (EDGP). For the EDGP, Eq.~\eqref{dgp} is often reformulated to:
\begin{equation}
  \forall \{i,j\}\in E\qquad \|x_i-x_j\|^2_2=d_{ij}^2,\label{edgp}
\end{equation}
which is a system of quadratic polynomial equations with no linear terms \cite[\S 2.4]{dgbook}.

The EDGP is motivated by many scientific and technological applications. The clock synchronization problem, for example, aims at establishing the absolute time of a set of clocks when only the time difference between subsets of clocks can be exchanged \cite{singer4}. The sensor network localization problem aims at finding the positions of moving wireless sensor on a 2D manifold given an estimation of some of the pairwise Euclidean distances \cite{eren04,eren06,snl}. The {\sc Molecular DGP} (MDGP) aims at finding the positions of atoms in a protein, given some of the pairwise Euclidean distances \cite{dgp-sirev,dgbook}. The position of autonomous underwater vehicles cannot be determined via GPS (since the GPS signal does not reach under water), but must rely on distances estimated using sonars: a DGP can then be solved in order to localize the fleet \cite{bahr}. Applications of the DGP to data science are described in \cite{dgds}; see \cite{dg4wv} for an application to natural language processing. In general, the DGP is an inverse problem that occurs every time one can measure some of the pairwise distances in a set of entities, and needs to establish their position.


The DGP is weakly {\bf NP}-hard even when restricted to simple cycle graphs (by reduction from {\sc Partition}) and strongly {\bf NP}-hard even when restricted to integer edge weights in $\{1,2\}$ in general graphs (by reduction from {\sc 3sat}) \cite{saxe79}. It is in {\bf NP} if $K=1$ but not known to be in {\bf NP} if $K>1$ for general graphs \cite{dgpinnp}, which is an interesting open question \cite{openprobsDG}.

There are many approaches to solving the DGP. Generally speaking, application-specific solution algorithms exploit some of the graph structure, whenever it is induced by the application. For example, a condition often asked when reconstructing the positions of sensor networks is that the realization should be unique (as one would not know how to choose between multiple realizations), a condition called {\it global rigidity} \cite{connelly}. This condition can, at least generically, be ensured by a specific graph rigidity structure of the unweighted input graph, as shown in \cite{gortler}. For protein structures, on the other hand, which are found in nature in several isomers, one is sometimes interested in finding all (incongruent) realizations of the given protein graph \cite{dmdgp,symmBPjbcb,liberti-gsi13}. Since such graphs are rigid, one can devise an algorithm (called Branch-and-Prune) that, following a given vertex order, branches on reflections of the position of the next vertex, which is computed using trilateration \cite{dgbook}. It is also possible that DGP problems arise in their full generality, i.e.~independently of any further knowledge on their structure or properties: for such cases, one can resort to Mathematical Programming (MP) formulations and corresponding solvers \cite{mdgpsurvey,zoo,isco16}.

The MP formulation that is most often used reformulates Eq.~\eqref{edgp} to the minimization of the sum of squared error terms:
\begin{equation}
  \min_x \sum\limits_{\{i,j\}\in E} (\|x_i-x_j\|_2^2 - d_{ij}^2)^2.  \label{go}
\end{equation}
This formulation describes an unconstrained polynomial minimization problem. The polynomial in question has degree 4, is always nonnegative, and generally nonconvex and multimodal. The decision variables are represented by a $n\times K$ rectangular matrix $x$ such that $x_{ik}$ is the $k$-th component of the vector $x_i$, which gives the position in $\mathbb{R}^K$ of vertex $i\in V$. Each solution $x^\ast\in\mathbb{R}^{nK}$ having global minimum value equal to zero is a realization of the given graph. Solutions with small objective function value represent approximate solutions. Because of the nonconvexity of the formulation and the hardness of the problem, Eq.~\eqref{go} is not usually solved to guaranteed $\varepsilon$-optimality (e.g.~using a spatial Branch-and-Bound approach \cite{couenne}); rather, heuristic approaches, such as MultiStart (MS) \cite{lln1}, Variable Neighbourhood Search (VNS) \cite{dvnsjogo}, or relaxation-based heuristics \cite{isco16,barvinok_orl} may be used.

As far as we know, all existing MP formulations for the EDGP are edge-based, such as the one in Eq.~\eqref{go}. In this paper we discuss a new MP formulation for the EDGP based on the incidence of cycles and edges instead, a relaxation based on Eulerian cycles, and a computational comparison with Eq.~\eqref{go}.

\leorev{Although this paper is not about graph theory, a fair amount of graph theoretical content is needed to prove the main reformulation result. Since the OJMO readership is supposed to be well versed in optimization but not necessarily in graph theory, we strove to achieve clarity and self-containment at the expense of compactness.}

\section{Some existing MP formulations}
\label{s:existing}
In this short section we give a minimal list of typical variants of Eq.~\eqref{go} in order to motivate the claim that the cycle-based formulation of the DGP discussed in this paper is new. Of course, only a complete enumeration of DGP formulations in the literature could substantiate this claim. But even this short list shows that the typical modelling approach for the DGP is direct: namely, decision variables encode the realization of each vertex as a vector in $\mathbb{R}^K$. Many more formulations of the DGP and its variants, all corresponding to this criterion, are given in \cite{lln1,mdgpsurvey,zoo}.

The closest variant of Eq.~\eqref{go} simply adds a constraint ensuring that the centroid of all of the points in the realization is at the origin \leorev{(see Lemma~\ref{lem:centroid} below)}. This removes the degrees of freedom given by translations:
\begin{equation}
  \left.\begin{array}{rl}
    \min\limits_x & \sum\limits_{\{i,j\}\in E} (\|x_i-x_j\|_2^2 - d_{ij}^2)^2 \\
    \forall k\le K & \sum\limits_{i\in V} x_{ik} = 0.
  \end{array}\right\} \label{go1}
\end{equation}
This formulation describes a linearly constrained polynomial minimization problem. Like Eq.~\eqref{go}, the polynomial in Eq.~\eqref{go1} has degree 4, is always nonnegative, and is generally nonconvex and multimodal.

Another small variant of Eq.~\eqref{go1} is achieved by adding range bounds to the the realization variables $x$; generally valid (but slack) bound values can be set to $\pm\frac{1}{2}\sum_{\{i,j\}\in E} d_{ij}$. This corresponds to the worst case of a single path being arranged in a straight line with unknown orientation.

Another possible formulation, derived again from Eq.~\eqref{go}, is obtained by replacing the squared error with absolute value errors (whose positive and negative parts are encoded by $s^+,s^-$). This yields the following formulation:
\begin{equation}
  \left.\begin{array}{rl}
    \min\limits_{s,x} & \sum\limits_{\{i,j\}\in E} (s^+_{ij} + s^-_{ij}) \\
    \forall \{i,j\}\in E & \|x_i-x_j\|_2^2 = d_{ij}^2 + s^+_{ij} - s^-_{ij}\\
    \forall \{i,j\}\in E & s^+_{ij}, s^-_{ij} \ge 0.
  \end{array}\right\}
  \label{minabs}
\end{equation}
Note that, again, each solution $s^\ast,x^\ast$ with zero optimal objective value makes $x^\ast$ an encoding of a realization of the given graph. Thus, global optima are preserved by this reformulation, while local optima may differ.

Yet another reformulation derived from replacing squared errors with absolute values consists in observing that the ``plus'' and ``minus'' parts of each absolute value term correspond to a convex and concave function. This yields a formulation called {\it push-and-pull}, since the objective pulls adjacent vertices apart, while the constraint push them back together:
\begin{equation}
  \left.\begin{array}{rrcl}
    \max\limits_x & \sum\limits_{\{i,j\}\in E} \|x_i-x_j\|_2^2 && \\
    \forall \{i,j\}\in E & \|x_i-x_j\|_2^2 &\le& d_{ij}^2.
  \end{array}\right\}
  \label{pushpull}
\end{equation}
Eq.~\eqref{pushpull} is a Quadratically Constrained Quadratic Program with concave objective and convex constraints. It was used within a Multiplicative Weights Update algorithm for the DGP in \cite{zoo}, as well as a basis for Semidefinite Programming and Diagonally Dominant Programming relaxations \cite{isco16,barvinok_orl}. It can be shown that all constraints are active at global optima, which therefore correspond to realizations of the given graph \cite{mwu}.

\section{A new formulation based on cycles}
\label{newform}
In this section we propose a new formulation for the EDGP, based on the fact that the quantities $x_{ik}-x_{jk}$ sum up to zero over all edges of any cycle in the given graph for each dimensional index $k\le K$. This idea was used in \cite{saxe79} for proving weak {\bf NP}-hardness of the DGP on cycle graphs. For a subgraph $H$ of a graph $G=(V,E)$, we use $V(H)$ and $E(H)$ to denote vertex and edge set of $H$ explicitly; given a set $F$ of edges we use $V(F)$ to denote the set of incident vertices. Let $m=|E|$ and $n=|V|$. For a mapping $x:V\to\mathbb{R}^K$ we denote by $x[U]$ the restriction of $x$ to a subset $U\subseteq V$. Furthermore, we let a {\it closed trail} be a sequence of vertices and of the edges joining them, which begins and ends at the same vertex, and is such that no edge is repeated.
\begin{lemma}
  \label{lemcycle}
  Given an integer $K>0$, a simple undirected weighted graph $G=(V,E,d)$ and a mapping $x:V\to\mathbb{R}^K$, then for each cycle $C$ in $G$, each orientation of the edges in $C$ given by a closed trail $W(C)$ in the cycle, and each $k\le K$ we have:
  \begin{equation}
    \sum\limits_{(i,j)\in W(C)} (x_{ik}-x_{jk})=0. \label{sumcycles}
  \end{equation}
\end{lemma}
\begin{proof}
We renumber the vertices in $V(C)$ to $1,2,\ldots,\gamma=|V(C)|$ following the walk order in $W(C)$. Then Eq.~\eqref{sumcycles} can be explicitly written as:
  \begin{eqnarray*}
   (x_{1k}-x_{2k}) + (x_{2k}-x_{3k}) + \cdots + (x_{\gamma k}-x_{1k}) &=& \\
  =  x_{1k} - (x_{2k}-x_{2k}) - \cdots - (x_{\gamma k}-x_{\gamma k}) - x_{1k} &=& 0,
  \end{eqnarray*}
  as claimed. 
\end{proof}

We introduce new decision variables $y_{ijk}$ replacing the terms $x_{ik}-x_{jk}$ for each $\{i,j\}\in E$ and $k\le K$. Eq.~\eqref{edgp} then becomes:
\begin{equation}
  \forall\{i,j\}\in E\qquad \sum\limits_{k\le K} y_{ijk}^2 = d_{ij}^2. \label{newsys}
\end{equation}
We note that, with a slight abuse of notation, we index the sum in Eq.~\eqref{newsys} with the shorthand $k \le K$ instead of $k \in \{1, 2, \dots, K\}$. We will keep this notation throughout the paper, for ease of reading.
Moreover, we remark that for the DGP with other norms this constraint changes. For the $\ell_1$ or $\ell_\infty$ norms, for example, we would have:
\begin{equation}
  \forall\{i,j\}\in E\quad \sum\limits_{k\le K} |y_{ijk}| = d_{ij}\quad\mbox{ or }\quad\max\limits_{k\le K} |y_{ijk}| = d_{ij}. \label{linnorms}
\end{equation}

Next, we adjoin the constraints on cycles:
\begin{equation}
  \forall k\le K, \leorev{C\subseteq E}\quad \bigg(C\mbox{ is a cycle}\Rightarrow \sum\limits_{\{i,j\}\in E(C)} y_{ijk} = 0\bigg). \label{newcon}
\end{equation}

We also note that the feasible value of a $y_{ijk}$ variable is the (oriented) length of the segment representing the edge $\{i,j\}$ projected on the $k$-th coordinate. We can therefore infer bounds for $y$ as follows:
\begin{equation}
\forall k\le K, \{i,j\}\in E \quad -d_{ij} \le y_{ijk}\le d_{ij}.\label{ybnd}
\end{equation}
Although Eq.~\eqref{ybnd} are not necessary to solve the cycle formulation, they may improve performance of spatial Branch-and-Bound (sBB) algorithms \cite{tawarmalani1,couenne} and of various ``matheuristics'' \cite{recipeconf} that need explicit bounds on all variables, as well as allow an exact linearization of variable products, should a $y$ variable occur in a product with a binary variable in some DGP variant.

\leorev{We now give the following definition and state our main result, i.e., that Eq.~\eqref{newsys} and \eqref{newcon} are a valid MP formulation for the EDGP.
\begin{definition}
Given a strictly positive $K \in \mathbb{N}$ and a graph $G=(V, E)$, $Y \triangleq \{y \in \mathbb{R}^{Km}\;|\; \eqref{newsys} \land \eqref{newcon}\}$ is the set of vectors satisfying Eq.~\eqref{newsys} and \eqref{newcon}.
\label{def:Y}
\end{definition}
We emphasize that $Y$ depends on the EDGP instance $(K, G)$.
}

\begin{theorem}
\label{mainthm}
\leorev{The set $Y$ is non-empty} if and only if $(K,G)$ is a YES instance of the EDGP.
\end{theorem}

The proof argues by recursion on a graph decomposition of $G$ that a certain linear system related to the cycles of $G$ (see Eq.~\eqref{xysys} below) has a solution in the $x$ variables if and only if the given EDGP instance is YES, as certified by the $y$ variables\footnote{This is not the only way to construct $x$ from $y$: three colleagues, in three separate occasions, have suggested that path lengths (as measured by sums of $y$ variables) can yield valid values for the $x$ variables in each dimension: then, the cycle condition would prove consistency of $x$ and $y$. This is easy enough to explain informally. When we set about formalizing this suggestion, so that it would be clear in all its parts, we realized that the proof would likely be as long as the one we present here.}.

We shall construct our proof by steps. The first step defines a graph decomposition based on the removal of a single vertex. Given a graph $G=(V,E)$ and a subset $U\subset V$, the subgraph $G[U]$ {\it induced} by $U$ is the graph $(U,\{\{u,v\}\in E\;|\;u,v\in U\})$. With a slight abuse of notation we denote the vertices of a graph $G'$ by $V(G')$ and its edges by $E(G')$. We let $\gamma(G)$ be the number of connected components of $G$. A vertex $v$ of $G$ with the property that $\gamma(G[V\smallsetminus\{v\}])>\gamma(G)$ is called a {\it cut vertex}. A graph $G$ is {\it biconnected} if, for any pair $u,v$ of distinct vertices of $G$, there is a simple cycle in $G$ incident to $u$ and $v$. It is not hard to show that biconnectedness is equivalent to connectedness and the absence of cut vertices. \leorev{To see this, we first introduce the concept of ``1-decomposition'', then prove some statements related to it.}
\begin{definition}
\label{1dec}
A {\it $1$-decomposition} of a graph $G=(V,E)$ is a set of subgraphs $G_1,\ldots,G_r$ (where $r\in\mathbb{N}$ with $r\ge 1$) of $G$ such that:
\begin{enumerate}[(a)]
\item $G_i$ is either biconnected or a tree for all $i\le r$;\label{1da}
\item $\bigcup_{i\le r}E(G_i)=E$;\label{1db}
\item for any $i<j\le r$ the intersection $V(G_i)\cap V(G_j)$ is either empty or it consists of a single cut vertex of $G$. \label{1dc}
\end{enumerate}
A $1$-decomposition of $G$ is {\it nontrivial} if $r>1$. A graph $G$ is $1$-decomposable if it has a nontrivial $1$-decomposition.
\end{definition}

\leorev{The $1$-decomposition bears some relationship to the block-cutpoint tree defined by Harary in \cite[p.~36]{harary}. However, subgraphs in the $1$-decomposition may also be trees, which cannot appear in Harary's construction, since every vertex of a tree is a cutpoint by definition. Trees are important because they are easy to realize in $\mathbb{R}^K$. Their realizations can then be paste to the realizations of the other subgraphs by rotations and translations, a fact that is used in the proof of the main theorem. The same would not follow if we were to use Harary's block-cutpoint trees, since they contract blocks to a single vertex. We do, however, invoke \cite[Thm.~3.1]{harary} to state that a connected graph $G=(V,E)$ is $1$-decomposable if and only if it has a cut vertex.}

\begin{lemma}
Let $G$ be $1$-decomposable, with decomposition $\mathcal{G}=\{G_1,\ldots,G_r\}$, and $C$ be a cycle in $G$. Then there is an index $i\le r$ s.t.~$C$ is a subgraph of $G_i$.\label{lem01a2}
\end{lemma}
\begin{proof}
Suppose, to aim at a contradiction, that there are two distinct subgraphs $G_i,G_j$ in $\mathcal{G}$ both incident to the edges of $C$. Then there is a nontrivial path $p$ in $C$, with at least two edges, joining a vertex $u$ in $G_i$ to a vertex $v$ in $G_j$. Therefore, \leorev{by \cite[Thm.~3.1]{harary},} there must be a cut vertex of $G$ on $p$, which implies that there is a cut vertex in $C$, which is impossible, since cycles are biconnected. 
\end{proof}
We note that no biconnected graph $G$ is $1$-decomposable. On the other hand, a tree with $n$ vertices can always be 1-decomposed into $n$ subgraphs.

\begin{proposition}
Any connected component $G=(V,E)$ of a simple graph has a (possibly trivial) $1$-decomposition consisting of biconnected subgraphs and tree subgraphs. \label{lem01d}
\end{proposition}
\begin{proof}
We prove this result by induction on the number $\beta$ of biconnected subgraphs in a $1$-decomposition $\mathcal{C}=\{G_1,\dots,G_r\}$ of $G$ for some $r\in\mathbb{N}$. We first deal with the base case, where $\beta=0$. We claim that $G$ must be a tree: supposing $G$ has a cycle $G'$, as well as biconnectedness of cycles and part (\ref{1dc}) of Defn.~\ref{1dec}, $G'$ must be one of the $G_1,\ldots,G_r$. But then $\beta\geq 1$ against the assumption. Therefore, the trivial $1$-decomposition $\mathcal{C}=\{G\}$ is a valid $1$-decomposition of $G$. We now tackle the induction step. Consider the largest biconnected subgraph $B$ of $G$: then $\tilde{G}=G[V\smallsetminus V(B)]$ has one fewer biconnected components than $G$, so, by induction, $\tilde{G}$ has a $1$-decomposition $\mathcal{D}'=\{G'_1,\ldots,G'_{t-1}\}$ for some $t\in\mathbb{N}$ with $t>1$. We prove that $\mathcal{D}=\mathcal{D}'\cup\{B\}$ is a valid $1$-decomposition of $G$. Condition (\ref{1da}) is verified since $\mathcal{D}'$ is a valid 1-decomposition by induction, and $B$ is biconnected; condition (\ref{1db}) is verified since the union of the graph in $\mathcal{D}$ is $\mathcal{G}$ by construction; for condition (\ref{1dc}), suppose there is $i<t$ s.t.~$|V(G_i) \cap V(B)|\ge 2$: this means there are two distinct vertices $u,v$ in both $V(G_i)$ and $V(B)$. Since $G_i$ is connected, there must be a path $p$ from $u$ to $v$ in $G_i$, hence $G[B\cup V(p)]$ is a biconnected graph larger than $B$. But $B$ was assumed to be largest, so this is not possible, and (\ref{1dc}) holds, which concludes the proof. 
\end{proof}

The second step proves the easier ($\Leftarrow$) direction of Thm.~\ref{mainthm}.
\begin{proposition}
  \label{rightdir}
For any YES instance $(K,G)$ of the EDGP there is a vector \leorev{$y^\ast \in Y$}.
\end{proposition}
\begin{proof}
Assume that $(K,G)$ is a YES instance of the EDGP. Then $G$ has a realization $x^\ast\in\mathbb{R}^{nK}$ in $\mathbb{R}^K$. We define $y_{ijk}^\ast=x^\ast_{ik}-x^\ast_{jk}$ for all $\{i,j\}\in E$ and $k\le K$. Since $x^\ast$ is a realization of $G$, by definition it satisfies Eq.~\eqref{edgp}, and, by substitution, Eq.~\eqref{newsys}. Moreover, any realization of $G$ satisfies Eq.~\eqref{sumcycles} over each cycle by Lemma \ref{lemcycle}. Hence, by replacement, it also satisfies Eq.~\eqref{newcon}. 
\end{proof}

In the third step, we lay the groundwork towards the more difficult ($\Rightarrow$) direction of Thm.~\ref{mainthm}. We proceed by contradiction: we assume that $(K,G)$ is a NO instance of the EDGP, and suppose that \leorev{the set $Y$ for this instance is non-empty}. For every $y\in Y$ we consider the $K$ linear systems
\begin{equation}
\forall \{i,j\}\in E\quad x_{ik}-x_{jk}=y_{ijk}, \label{xysys}
\end{equation}
for each $k\le K$, each with $n$ variables and $m$ equations. We square both sides then sum over $k\le K$ to obtain
\begin{equation}
  \forall \{i,j\}\in E\quad \sum_{k\le K} (x_{ik}-x_{jk})^2= \sum_{k\le K} y_{ijk}^2.\label{xysys1a}
\end{equation}
By Eq.~\eqref{newsys} we have
\begin{equation}
  \sum_{k\le K} y_{ijk}^2 = d_{ij}^2,\label{xysys1b}
\end{equation}
whence follows Eq.~\eqref{edgp}, contradicting the assumption that the EDGP is NO. So we only need to show that there is a solution $x^\ast$ to Eq.~\eqref{xysys} for any given $y\in Y$. To this effect, we shall exploit the $1$-decomposition of $G$ into biconnected graphs and trees derived in Prop.~\ref{lem01d}. First, though, we have to show that Eq.~\eqref{xysys} has a solution if $Y\not=\varnothing$ in the ``base cases'' of the $1$-decomposition, namely trees and biconnected graphs.

\leorev{The following result essentially proves that the constraint matrix of Eq.~\eqref{xysys} has full rank, which is an easy consequence of graphic matroid theory. We prove the result by elementary means for self-containment.}
\begin{lemma}
  Let $G=(V,E)$ be a tree, and $Y\not=\varnothing$. Then Eq.~\eqref{xysys} has a solution for every $k\le K$.
  \label{tree}
\end{lemma}
\begin{proof}
\leorev{Let $M$ be the coefficient matrix of the system of equations~\eqref{xysys}, for a given $k\le K$; and let $y^{k}$ be the vector $(y_{uvk}\;|\;\{u,v\}\in E)$. We note that, since $M$ is the (transposed) incidence matrix of $G$, only the right-hand side of the system changes for each $k$.} We aim at proving that $M$ and $(M,y^{k})$ have the same rank, and that this rank is full. We proceed by induction on the size $|E|$ of the tree. The base case, where $|E|=1$ and $G$ consists of a single edge $\{u,v\}$, yields $M=(1,-1)$ with rank 1 for each $k\le K$. By inspection, $(M,y_{uvk})$ also has rank 1 for any $y_{uvk}$. Consider a tree $G'$ with one fewer edge (say, $\{u,v\}$) than $G$, such that $V\smallsetminus V(G')=\{v\}$. Let the corresponding system Eq.~\eqref{xysys} $\tilde{M}x=\tilde{y}$ satisfy $\mathsf{rank}(\tilde{M})=\mathsf{rank}(\tilde{M},\tilde{y}^{k})$, for all $k\le K$. Then the shape of $M$ is:
  \begin{equation*}
    M = \left(\begin{array}{cc} \tilde{M} & 0 \\  e_u & -1\end{array}\right),
  \end{equation*}
  where $e_u=(0,\ldots,0,1_u,0,\ldots,0)$. This shows that $\mathsf{rank}(M)=\mathsf{rank}(\tilde{M})+1$, that this rank is full, and hence also that $\mathsf{rank}(M)=\mathsf{rank}((M,y^{k}))$. 
\end{proof}

\begin{lemma}
  Let $G=(V,E)$ be biconnected, and $Y\not=\varnothing$. Then Eq.~\eqref{xysys} has a solution for every $k\le K$. \label{biconn}
\end{lemma}
\begin{proof}
We proceed by induction on the simple cycles of $G$. For the base case, we consider $G$ to be a graph consisting of a single cycle, with corresponding \leorev{$y \in Y$}. Since $G$ is a cycle, it has the same number of vertices and edges, say $q$. This implies that, for any fixed $k\le K$, Eq.~\eqref{xysys} is a linear system $Mx=y^k$ (where $y^k=(y_{uvk}\;|\;\{u,v\}\in E)$) with a $q\times q$ coefficient matrix:
\begin{equation}
M=\left(\begin{array}{ccccc}
  1 & -1 &    & & \\
    &  1 & -1 & & \\
    &    &  1 &\ddots & \\
    &    &    &\ddots & -1 \\
 -1 &    &    & & 1
\end{array}\right).
\label{eqM}
\end{equation}
\leorev{We remark that $M$ is the incidence matrix of $G$ as in the Proof of Lemma~\ref{tree}.} By Eq.~\eqref{sumcycles} and by inspection of Eq.~\eqref{eqM} it is clear that $\mathsf{rank}(M)=q-1$: then Eq.~\eqref{newcon} ensures that $\mathsf{rank}((M,y^k))=\mathsf{rank}(M)$, and therefore that Eq.~\eqref{xysys} has a solution.

We now tackle the induction step. The incidence vectors in $E$ of the cycles of any graph are a vector space of dimension $m-n+1$ over the finite field $\mathbb{F}_2=\{0,1\}$ \cite{seshu}. We consider a fundamental cycle basis $\mathcal{B}$ of $G$ (see Sect.~\ref{s:cycbas}). We assume that (a) $G'$ is a union of fundamental cycles in $\mathcal{B}'\subsetneq\mathcal{B}$, for which Eq.~\eqref{xysys} has a solution $x'$ by the induction hypothesis, and (b) that $C$ is another fundamental cycle in $\mathcal{B}\smallsetminus\mathcal{B}'$, with a solution $x^C$ of Eq.~\eqref{xysys} that exists by the base case. We aim at proving that Eq.~\eqref{xysys} has a solution for $G'\cup C$. Since $G$ is biconnected, the induction can proceed by ear decomposition \cite{eardecomp}, which means that $G'$ is also biconnected, and that $C$ is such that $E(G')\cap E(C)=F$ is a non-empty path in $G'$.

By Eq.~\eqref{newcon} applied to $C$, we have
\begin{equation}
  \forall k\le K \quad \sum\limits_{\{i,j\}\in C} y_{ijk} = 0.\label{thm1eq2}
\end{equation}
Since $x'$ satisfies Eq.~\eqref{xysys} by the induction hypothesis,
\begin{equation}
  \forall k\le K, \{i,j\}\in F \quad x'_{ik}-x'_{jk} = y_{ijk}.\label{thm1eq3}
\end{equation}
We replace Eq.~\eqref{thm1eq3} in Eq.~\eqref{thm1eq2}, obtaining
\begin{equation}
  \forall k\le K \quad \sum\limits_{\{i,j\}\in F} (x'_{ik}-x'_{jk}) = -\sum\limits_{\{i,j\}\in E(C)\smallsetminus F} y_{ijk}.\label{thm1eq4}
\end{equation}
Moreover, $x^C$ also satisfies Eq.~\eqref{xysys} over $C$, hence we can replace the right hand side of Eq.~\eqref{thm1eq4} with the corresponding terms in $x^C_{ik}-x^C_{jk}$ to get:
\begin{equation}
  \forall k\le K \quad \sum\limits_{\{i,j\}\in F} (x'_{ik}-x'_{jk}) + \sum\limits_{\{i,j\}\in E(C)\smallsetminus F} (x^C_{ik}-x^C_{jk}) = 0.\label{thm1eq5}
\end{equation}

We now fix $x'$, and aim at modifying $x^C$ so that: (a) $x^C$ matches $x'$ on $V(F)$, (b) the modified $x^C$ is still a solution of Eq.~\eqref{xysys} on $C$. We set $x^C_{ik}$ to $x'_{ik}$ for each $i\in V(F)$, and consider the resulting linear system Eq.~\eqref{xysys} given by $M$, as in Eq.~\eqref{eqM}, for each $k\le K$, where we assume without loss of generality that $V(F)=\{1,\ldots,r\}$ and $V(C)=\{r+1,\ldots,s\}$:
{\small
\begin{equation}
  \left.\begin{array}{rrrrrrcll}
    x'_{1k} &-\ x'_{2k} & & & & &= & y_{12k} & (1) \\
    & x'_{2k} & -\ x'_{3k} & & & &= & y_{23k}& (2) \\
    &\ddots&\ddots & & & &\vdots & \vdots & \vdots \\
  &  & x'_{rk} &-\ x^C_{r+1,k} & & & = & y_{r,r+1,k} & (r)\\
    & & & x^C_{r+1,k} &-\ x^C_{r+2,k} & & = & y_{r+1,r+2,k} & (r\!\!+\!\!1) \\
    & & & \ddots & \ddots & &\vdots &\vdots & \vdots \\
    & & & & x^C_{s-1,k} & -\ x^C_{sk} & = & y_{s-1,s,k} & (s\!\!-\!\!1) \\
   -\ x'_{1k} & & & & & x^C_{sk} & = & y_{1sk}. & (s) \\
  \end{array}\right\}\label{x'xC}
\end{equation}
}
The equations from ($1$) to ($r\!-\!1$) in Eq.~\eqref{x'xC} are satisfied by the induction hypothesis since they only depend on $x'$, so we can remove them from the system and assume $x'$ to be constant. We are left with:
\begin{equation}
  \left.\begin{array}{rrrcll}
  -\ x^C_{r+1,k} & & & = & y_{r,r+1,k}-x'_{rk} & (r)\\
    x^C_{r+1,k} &-\ x^C_{r+2,k} & & = & y_{r+1,r+2,k} & (r\!\!+\!\!1) \\
    \ddots & \ddots & &\vdots &\vdots & \vdots \\
    & x^C_{s-1,k} & -\ x^C_{sk} & = & y_{s-1,s,k} & (s\!\!-\!\!1) \\
    & & x^C_{sk} & = & y_{1sk}+x'_{1k}. & (s) \\
  \end{array}\right\}\label{x'xC1}
\end{equation}
Summing up the left hand sides of Eq.~\eqref{x'xC1}, we obtain:
\begin{eqnarray*}
  && -x^C_{r+1,k} + (x^C_{r+1,k}-x^C_{r+2,k}) + \cdots + (x^C_{s-1,k} - x^C_{sk}) + x^C_{sk} \\
  &=& (-x^C_{r+1,k} + x^C_{r+1,k}) +  \cdots + (-x^C_{sk} + x^C_{sk}) = 0
\end{eqnarray*}
for all $k\le K$, so the $(s-r+1)\times (s-r+1)$ matrix $\bar{M}$ of the $k$-th linear system Eq.~\eqref{x'xC1} has rank $\le s-r$. On the other hand, eliminating the first or last row makes it clear by inspection that the rest of the rows are linearly independent; therefore the rank of $\bar{M}$ is exactly $s-r$. Summing up the components of the right hand side vector $\bar{y}^k$ of Eq.~\eqref{x'xC1}, we obtain:
\begin{eqnarray*}
  \chi &=& - x'_{rk} + y_{r,r+1,k} + y_{r+1,r+2,k} + \cdots + y_{s-1,s,k} + y_{1sk} + x'_{1k} \\
  &=& (x'_{1k} - x'_{rk}) + \sum_{\{i,j\}\in E(C)\smallsetminus F} y_{ijk}.
\end{eqnarray*}
We remark that
\begin{eqnarray*}
  x'_{1k}-x'_{rk} &=& (x'_{1k}-x'_{2k})+(x'_{2k}-x'_{3k}) + \cdots + (x'_{r-1,k}+x'_{rk}) \\
  &=& \sum_{\{i,j\}\in F} (x'_{ik}-x'_{jk}) = \sum_{\{i,j\}\in F} y_{ijk}
\end{eqnarray*}
since $x'$ satisfies Eq.~\eqref{xysys} by the induction hypothesis. Therefore
\begin{equation*}
\chi = \sum_{\{i,j\}\in F} y_{ijk} + \sum_{\{i,j\}\in E(C)\smallsetminus F} y_{ijk} = \sum_{\{i,j\}\in E(C)} y_{ijk},
\end{equation*}
whence $\chi=0$ by Eq.~\eqref{thm1eq2}. This implies that $\mathsf{rank}((\bar{M},\bar{y}^k))=\mathsf{rank}(\bar{M})=s-r$. Therefore, Eq.~\eqref{x'xC1} has a solution, which yields the modified $x^C$ with properties (a) and (b) given above. This concludes the induction step and the proof. 
\end{proof}

We can finally give the proof of Thm.~\ref{mainthm}.

\noindent {\it Proof of Thm.~\ref{mainthm}.} The ($\Leftarrow$) part follows by Prop.~\ref{rightdir}. For the ($\Rightarrow$) part, we exploit a $1$-decomposition of $G$ into trees and biconnected subgraphs, derive solutions to Eq.~\eqref{xysys} for each subgraph, and show that the solutions can be easily combined to yield a solution to Eq.~\eqref{xysys} for the whole graph $G$.

We assume without loss of generality that $G$ is connected (otherwise each connected component can be treated separately), and consider  a $1$-decomposition $\mathcal{D}=\{G_1,\ldots,G_r\}$ of $G$. By Lemmata \ref{tree} and \ref{biconn}, there exist solutions $x^1,\dots, x^r$ to Eq.~\eqref{xysys} applied to $G_1,\ldots,G_r$ respectively. Consider the graph
\[\mathscr{D}=(\mathcal{D},\; \{\{i,j\}\;|\; 1\le i\not=j\le r\land |V(G_i)\cap V(G_j)|=1\}).\]
By Lemma \ref{lem01a2}, $\mathscr{D}$ is a tree: otherwise, a cycle in $\mathcal{D}$ would be a contraction of a cycle in $G$ not included in a single $G_i$, against Lemma \ref{lem01a2}. This allows us to reorder $\mathcal{D}$ so that, for each $j>1$, there is a unique $i<j$ such that $\{i,j\}\in E(\mathscr{D})$.

We remark that, for each $i\le r$, $x^i$ is a realization of $G_i$ in $\mathbb{R}^K$ by Eq.~\eqref{xysys}-\eqref{xysys1b}. More precisely, $x^i$ is a $|V(G_i)|\times K$ matrix $x^i=(x^i_{\ell k})$ so that $x^i_\ell=(x^i_{\ell 1},\ldots,x^i_{\ell K})$ is the position of vertex $\ell\in V(G_i)$ in $\mathbb{R}^K$. Note that the realizations $x^1,\ldots,x^r$ can be modified by translations without changing the values of $y$ (by inspection of Eq.~\eqref{xysys}).

We now construct a solution $\bar{x}$ of Eq.~\eqref{xysys} for $G$ by induction on $\mathcal{D}$ ordered as described above. For the base case $i=1$, we fix $x^1$ in any way (e.g.~by taking the centroid of the rows of $x^1$ to be the origin), and initialize the first $|V(G_1)|$ rows of $\bar{x}$ with those of $x^1$. For any $i>1$, we identify the unique predecessor $j$ of $i$ in the order on $\mathcal{D}$. The induction hypothesis ensures the existence of a solution $\bar{x}$ of the union of $G_1,\ldots,G_j$. Consider the cut vertex $v$ in $V(G_j)\cap V(G_i)$ guaranteed by definition of the order on $\mathcal{D}$, and let $\bar{x}_v\in\mathbb{R}^K$ be its position. Then the translation $\tilde{x}^i=x^i-\mathbf{1}\transpose{(x^i_{v}-\bar{x}_v)}$ yields another valid solution of Eq.~\eqref{xysys} applied to $G^i$ by translation invariance, and this solution is such that $\tilde{x}^i_{v} = \bar{x}_v$. Therefore, using the rows of $\tilde{x}^i$, $\bar{x}$ can be extended to a solution of Eq.~\eqref{xysys} applied to the union of $G_1,\ldots,G_j$ and $G^i$, as claimed. 

Thm.~\ref{mainthm} can also be interpreted as a polynomial reduction of the EDGP to the problem of finding a solution of Eq.~\eqref{newsys} and \eqref{newcon}.
\begin{corollary}
  Deciding feasibility of Eq.~\eqref{newsys} and \eqref{newcon} is $\mathbf{NP}$-hard.
  \label{nphard}
\end{corollary}
\begin{proof}
  By reduction from EDGP using Thm.~\ref{mainthm}. 
\end{proof}

A remarkable consequence of Thm.~\ref{mainthm} is that it allows a decomposition of the computation of the realization $x$ into two stages: first, solve Eq.~\eqref{newsys}-\eqref{newcon} to find a feasible $y^\ast$; then solve
\begin{equation}
  \forall k\le K, \{i,j\}\in E\quad x_{ik}-x_{jk} = y_{ijk}^\ast \label{xysys2}
\end{equation}
to find a realization $x^\ast$. We note that Eq.~\eqref{xysys2} is just a restatement of Eq.~\eqref{xysys} universally quantified over $k$.
\begin{corollary}
  \leorev{Given an EDGP instance $(K, G)$ and a solution $y^\ast \in Y$, any solution $x^\ast$ of Eq.~\eqref{xysys2} is a valid realization of the given instance.}
  \label{2stage}
\end{corollary}
\begin{proof}
  The feasibility of Eq.~\eqref{xysys2} with the right hand side replaced \leorev{by $y^\ast \in Y$} follows directly from Thm.~\ref{mainthm}, since if such a $y^\ast$ exists then the EDGP is feasible. 
\end{proof}
The first stage is $\mathbf{NP}$-hard by Cor.~\ref{nphard}, while the second stage is tractable, since solving linear systems can be done in polynomial time.

\begin{remark}
  \label{xysysfeas}
  Note that Eq.~\eqref{xysys2} has $Km$ equations, but its rank may be lower, since there are only $Kn$ variables: in particular, Eq.~\eqref{xysys2} may be an overdetermined linear system. The feasibility of this system is guaranteed by Cor.~\ref{2stage}; in particular, the steps of the proof of Thm.~\ref{mainthm} imply that Eq.~\eqref{xysys2} loses rank w.r.t.~$Km$ according to the incidence of the edges in the cycles of $G$. In other words, any solution $y'$ to Eq.~\eqref{newcon} provides a right hand side to Eq.~\eqref{xysys2} that makes the system feasible.
\end{remark}

The issue with Thm.~\eqref{mainthm} is that it relies on the exponentially large family of constraints Eq.~\eqref{newcon}. While this is sometimes addressed by algorithmic techniques such as row generation, we shall see in the following that it suffices to consider a polynomial set of cycles (which, moreover, can be found in polynomial time) in the quantifier of Eq.~\eqref{newcon}.

\section{The cycle vector space and its bases}
\label{s:cycbas}
We recall that incidence vectors of cycles (in a Euclidean space having $|E|$ dimensions) form a vector space over a field $\mathbb{F}$, which means that every cycle can be expressed as a weighted sum of cycles in a basis. In this interpretation, a {\it cycle} in $G$ is simply a subgraph of $G$ where each vertex has even degree: we denote their set by $\mathcal{C}$. This means that Eq.~\eqref{newcon} is actually quantified over a subset of $\mathcal{C}$, namely the simple connected cycles. Every basis has cardinality $m-n+a$, where $a$ is the number of connected components of $G$. If $G$ is connected, cycle bases have cardinality $m-n+1$ \cite{seshu}.

Our interest in introducing cycle bases is that we would like to quantify Eq.~\eqref{newcon} polynomially rather than exponentially in the size of $G$. Our goal is to replace ``$C$ is any simple connected cycle in $\mathcal{C}$'' by ``$C$ is a cycle in a cycle basis of $G$''. In order to show that this limited quantification is enough to imply every constraint in Eq.~\eqref{newcon}, we have to show that, for each simple connected cycle $C\in\mathcal{C}$, the corresponding constraint in Eq.~\eqref{newcon} can be obtained as a weighted sum of constraints corresponding to the basis elements.

Another feature of Eq.~\eqref{newcon} to keep in mind is that edges are implicitly given a direction: for each cycle, the term for the {\it undirected} edge $\{i,j\}$ in Eq.~\eqref{newcon} is $(x_{ik}-x_{jk})$. Note that while $\{i,j\}$ is exactly the same vertex set as $\{j,i\}$, the corresponding term is either positive or not, depending on the direction $(i,j)$ or $(j,i)$. We deal with this issue by arbitrarily directing the edges in $E$ to obtain a set $A$ of arcs, and considering {\it directed} cycles in the directed graph $\bar{G}=(V,A)$. In this interpretation, the incidence vector of a directed cycle $C$ of $\bar{G}$ is a vector $c^C\in\mathbb{R}^m$ satisfying \cite[\S 2, p.~201]{mehlhorn_cb}:
\begin{equation}
  \forall j\in V(C)\quad \sum\limits_{(i,j)\in A} c^C_{ij} = \sum\limits_{(j,\ell)\in A} c^C_{j\ell}.\label{flowcons}
\end{equation}

A directed circuit $D$ of $\bar{G}$ is obtained by applying the edge directions from $\bar{G}$ to a connected subgraph of $G$ where each vertex has degree exactly 2 (note that a directed circuit need not be strongly connected, although its undirected version is connected). Its incidence vector $c^D\in\{-1,0,1\}^m$ is defined as follows:
\begin{equation*}
  \forall (i,j)\in A \qquad c^D_{ij} \triangleq \left\{\begin{array}{rcl} 1 & \mbox{if} & (i,j)\in A(D) \\
  -1 & \mbox{if} & (j,i)\in A(D) \\
  0 & \mbox{otherwise} & \end{array}\right.
\end{equation*}
where we have used $A(D)$ to mean the arcs in the subgraph $D$. In other words, whenever we walk over an arc $(i,j)$ in the natural direction $i\to j$ we let the $(i,j)$-th component of $c^D$ be $1$; if we walk over $(i,j)$ in the direction $j\to i$ we assign a $-1$, and otherwise a zero.

\subsection{Constraints over cycle bases}
\label{ccycbas}
The properties of undirected and directed cycle bases have been investigated in a sequence of papers by many authors, culminating with \cite{mehlhorn_cb}. We now prove that it suffices to quantify Eq.~\eqref{newcon} over a directed cycle basis.
\begin{proposition}
  \label{propnewcon}
  Let $\mathcal{B}$ be a directed cycle basis of $\bar{G}$ over $\mathbb{Q}$. Then Eq.~\eqref{newcon} holds if and only if:
\begin{equation}
  \forall k\le K, B\in\mathcal{B}\qquad \sum\limits_{(i,j)\in A(B)} c^B_{ij}y_{ijk} = 0.\label{newcon2}
\end{equation}
\end{proposition}
\begin{proof}
Necessity $\eqref{newcon}\Rightarrow\eqref{newcon2}$ follows because Eq.~\eqref{newcon} is quantified over all cycles: in particular, it follows for any undirected cycle in any undirected cycle basis. Moreover, the signs of all terms in the sum of Eq.~\eqref{newcon2} are consistent, by definition, with the arbitrary edge direction chosen for $\bar{G}$. \\ [0.3em]
Next, we claim sufficiency $\eqref{newcon2}\Rightarrow\eqref{newcon}$. Let $C\in\mathcal{C}$ be a simple cycle, and $\bar{C}$ be its directed version with the directions inherited from $\bar{G}$. Since $\mathcal{B}$ is a cycle basis, we know that there is a coefficient vector $(\gamma_B\;|\;B\in\mathcal{B})\in\mathbb{R}^{|\mathcal{B}|}$ such that:
  \begin{equation}
    c^{\bar{C}} = \sum\limits_{B\in\mathcal{B}} \gamma_B c^B. \label{eqcC}
  \end{equation}
  We now consider the expression:
  \begin{equation}
    \forall k\le K\quad \sum\limits_{B\in\mathcal{B}} \gamma_B\sum\limits_{(i,j)\in A(B)} c^B_{ij} y_{ijk}.\label{eqcB}
  \end{equation}
  On the one hand, by Eq.~\eqref{eqcC}, Eq.~\eqref{eqcB} is identically equal to $\sum_{(i,j)\in A(\bar{C})} c^{\bar{C}}_{ij}y_{ijk}$ for each $k\le K$; on the other hand, each inner sum in Eq.~\eqref{eqcB} is equal to zero by Eq.~\eqref{newcon2}. This implies $\sum_{(i,j)\in A(\bar{C})} c^{\bar{C}}_{ij} y_{ijk} = 0$ for each $k\le K$. Since $C$ is simple and connected, $\bar{C}$ is a directed circuit. This implies that $c^{\bar{C}}\in\{-1,0,1\}$. Now it suffices to replace $-y_{ijk}$ with $y_{jik}$ to obtain
  \[\forall k\le K\quad \sum\limits_{\{i,j\}\in E(C)} y_{ijk} = 0, \]
  where the edges on $C$ are indexed in such a way as to ensure they appear in order of consecutive adjacency. 
\end{proof}
Obviously, if $\mathcal{B}$ has minimum (or just small) cardinality, Eq.~\eqref{newcon2} will be sparsest (or just sparse), which is often a desirable property of linear constraints occurring in MP formulations. Hence we should attempt to find short cycle bases $\mathcal{B}$.

In summary, given a basis $\mathcal{B}$ of the directed cycle space of $\bar{G}$ where $c^B$ is the incidence vector of a cycle $B\in\mathcal{B}$, the following:
\begin{equation}
  \left.
  \begin{array}{rrcl}
    \min\limits_{s\ge 0, y} & \sum\limits_{\{i,j\}\in E} (s_{ij}^+ + s^-_{ij}) && \\
  \forall (i,j)\in A(\bar{G}) &  \sum\limits_{k\le K} y_{ijk}^2 - d_{ij}^2 &=& s^+_{ij} - s^-_{ij} \\
  \forall k\le K, B\in\mathcal{B} & \sum\limits_{(i,j)\in A(B)} c^B_{ij}y_{ijk} &=& 0
  \end{array}\right\}\label{dgpnew}
\end{equation}
is a valid formulation for the EDGP. The solution of Eq.~\eqref{dgpnew} yields a feasible vector $y^\ast$. As pointed out in Cor.~\ref{2stage}, we must then solve Eq.~\eqref{xysys2} to obtain a realization $x^\ast$ for $G$.

\subsection{How to find directed cycle bases}
\label{s:findcycbas}
We require directed cycle bases over $\mathbb{Q}$. By \cite[Thm.~2.4]{mehlhorn_cb}, each undirected cycle basis gives rise to a directed cycle basis (so it suffices to find a cycle basis of $G$ and then direct the cycles using the directions in $\bar{G}$). Horton's algorithm \cite{horton87} and its variants \cite{golynski,liebchenrizzi} find a minimum cost cycle basis in polynomial time. The most efficient deterministic variant is $O(m^3n)$ \cite{liebchenrizzi}, and the most efficient randomized variant has the complexity of matrix multiplication. Existing approximation algorithms have marginally better complexity.

It is not clear, however, that the provably sparsest constraint system will make the DGP actually easier to solve. We therefore consider a much simpler algorithm: starting from a spanning tree, we pick the $m-n+1$ circuits that each {\it chord} (i.e., non-tree) edge defines with the rest of the tree. This algorithm \cite{paton} yields a {\it fundamental} cycle basis (FCB). Finding the minimum FCB is known to be {\bf NP}-hard \cite{deo82}, but heuristics based on spanning trees prove to be very easy to implement and work reasonably well \cite{deo82} (optionally, their cost can be improved by an edge-swapping phase \cite{fcbmmor,fcbmatroid}).

\section{The Eulerian cycle relaxation}
\label{s:eulerian}
In this section we construct a relaxation of Eq.~\eqref{dgpnew}. This is accomplished by substituting the \leorev{$K|\mathcal{B}|$} cycle base constraints in Eq.~\eqref{newcon2} --- occurring as the last line in Eq.~\eqref{dgpnew} --- with \leorev{the $K$ constraints} obtained by considering a single Eulerian circuit in the given graph.

We follow a standard construction in order to find a Eulerian circuit, see e.g.~\cite{jungnickel}. We let $G'$ be the multigraph obtained from $G$ by adding sufficiently many parallel edges to $G$, so that the degree of each vertex in $G'$ is even. This can always be done by \cite{edmonds2}, which implies that $G'$ is Eulerian, i.e.~it has a cycle incident with every edge in $G'$ exactly once. We let $\mathscr{E}$ be a Eulerian cycle in $G'$, and let $\bar{\mathscr{E}}$ be either of the two orientations of $\mathscr{E}$ obtained by walking over the cycle. We let $\bar{G}'$ be the digraph induced by the Eulerian circuit $\bar{\mathscr{E}}$. For each $\{i,j\}\in E$ let $H_{ij}$ be the number of parallel edges between $i,j$ in $G'$.

We note that $\bar{G}'$ might have parallel and antiparallel arcs. Consider the family of arc subsets $\mathcal{H}_{ij}=\{(i',j',h)\;|\;h\le H_{ij}\land \{i',j'\}=\{i,j\}\}$ of $A(\bar{G}')$. We replace each arc $(i',j',h)\in\mathcal{H}_{ij}$ having $h>1$ by an oriented $2$-path $p_{i'j'h}=\{(i',v_{ijh}),(v_{ijh},j')\}$ involving a new added vertex $v_{ijh}$. We call $\tilde{G}$ the digraph obtained from $\bar{G}'$ with this replacement. We remark that $\tilde{G}$ is simple (it has no parallel/antiparallel arcs) by construction. Moreover, $\tilde{G}$ is a Eulerian digraph: take the Eulerian circuit $\bar{\mathscr{E}}$ in $\bar{G}'$, and, every time it traverses a parallel/antiparallel arc $(i',j',h)\in\mathcal{H}_{ij}$ with $h>1$, let it traverse the oriented $2$-path replacement $p_{i'j'h}$ instead: this is clearly a Eulerian circuit in $\tilde{G}$, which we call $\mathscr{C}$.

Next we consider the simple graph $\hat{G}$ obtained by replacing each arc in $\tilde{G}$ with an \leorev{(undirected)} edge. Let $\hat{V}=\{v_{ijh}\;|\;\{i,j\}\in E\land h>1\}$, and $\hat{E}$ be the subset of $E(\hat{G})$ obtained by losing the orientation of the arcs in
\[
\bigcup\limits_{\substack{(i',j',h)\in \mathcal{H}_{ij}\\ \{i,j\}\in E\land h>1}} p_{i'j'h}\,,
\]
i.e., the union of all the \leorev{edges} from the $2$-path replacements.
We note that, by construction,
\begin{equation}
	\hat{V}=V(\hat{G})\smallsetminus V\quad\land\quad
	\hat{E}=E(\hat{G})\smallsetminus E.\label{EVhat}
\end{equation}

Let $c^{\mathscr{C}}_{ij}\in\{1,-1\}$ be the orientation of $(i,j)$ in $\mathscr{C}$ w.r.t.~$\tilde{G}$; let $\hat{\mathscr{C}}$ be the simple Eulerian cycle in $\hat{G}$ corresponding to $\mathscr{C}$.

We can now prove the main result of this section.
\begin{proposition}
  \label{prop:dgprel}
  The formulation
\begin{equation}
  \left.
  \begin{array}{rrcl}
    \min\limits_{s\ge 0, y} & \sum\limits_{\{i,j\}\in E} (s_{ij}^+ + s^-_{ij}) && \\
  \forall (i,j)\in A(\tilde{G}) &  \sum\limits_{k\le K} y_{ijk}^2 - d_{ij}^2 &=& s^+_{ij} - s^-_{ij} \\
  \forall k\le K & \sum\limits_{(i,j)\in \mathscr{C}} c^{\mathscr{C}}_{ij}y_{ijk} &=& 0\quad(\dag)
  \end{array}\right\}\label{dgprel}
\end{equation}
is a relaxation of Eq.~\eqref{dgpnew}.
\end{proposition}
\begin{proof}
We first consider a variant of the cycle formulation in Eq.~\eqref{dgpnew} applied to $\hat{G}$, where, from the constraints corresponding to Eq.~\eqref{newsys} (second line of Eq.~\eqref{dgpnew}), we omit those indexed by $\hat{E}$. We call this variant ($\star$). We \underline{claim} that ($\star$) is an exact reformulation of Eq.~\eqref{dgpnew} applied to $G$. The claim holds because $E(\hat{G})\smallsetminus\hat{E}=E$ by Eq.~\eqref{EVhat}, and because the signs of the $y$ variables are irrelevant in Eq.~\eqref{newsys} since they are squared. Now, since $\hat{\mathscr{C}}$ is a Eulerian cycle in $\hat{G}$, Eq.~($\dag$) must hold in $\tilde{G}$ for any orientation of the edges of $\mathscr{C}$, by Lemma~\ref{lemcycle}. Therefore, Eq.~($\dag$) is an aggregation of the constraints in Eq.~\eqref{newcon2}, which occur within the reformulation ($\star$). So Eq.~\eqref{dgprel} is a relaxation of ($\star$). The proposition follows because of the \underline{claim}. 
\end{proof}

Note that Eq.~\eqref{dgprel} provides a solution $\bar{y}$ that may not satisfy Eq.~\eqref{newcon2}, which also guarantee feasibility in Eq.~\eqref{newcon} by Prop.~\ref{propnewcon}. By Remark \ref{xysysfeas}, this implies that Cor.~\ref{2stage} is no longer applicable. In other words, \leorev{we cannot obtain a realization $x$ of $G$ from $\bar{y}$ using the linear system} in Eq.~\eqref{xysys2}, since $\bar{y}$ might well make Eq.~\eqref{xysys2} infeasibile. \leorev{We can fix this issue by adjoining Eq.~\eqref{xysys2} to Eq.~\eqref{dgprel} as additional constraints. For practical reasons we also propose to adjoin the \textit{centroid constraints}
\begin{equation}
\left.
\begin{array}{rcl}
\forall k \leq K && \sum\limits_{i\in V} x_{ik} = 0\,,
\label{centroid}
\end{array}\right.
\end{equation}
which provide a restriction of Eq.~\eqref{dgpnew} by only keeping realizations of $G$ having zero centroid (see Eq.~\eqref{go1}).}

\leorev{For a formulation $P$, we denote by $\val{P}$ its optimal objective function value.
\begin{lemma}
  \label{lem:xysys2}
  Let $P$ be Eq.~\eqref{dgpnew}, and $P'$ be $P$ with the $x$ variables and the constraints in Eq.~\eqref{xysys2} adjoined. Then $\val{P}=\val{P'}$.
\end{lemma}
\begin{proof}
This is a direct consequence of Cor.~\ref{2stage}.
\end{proof}
\begin{lemma}
\label{lem:centroid}
For any reformulation (or relaxation) $P$ of the EDGP involving the $x$ variables, let $P'$ be $P$ with the centroid constraints Eq.~\eqref{centroid} adjoined. Then $\val{P}=\val{P'}$.
\end{lemma}
\begin{proof}
  Since $P'$ is a restriction of $P$, and the optimization direction is minimization, we have $\val{P}\le\val{P'}$. Let $x$ be an optimal solution of $P$: then $x'=x-\mathsf{stack}(\centroid{x},n)$ (where the second term of the right hand side is the centroid row $K$-vector stacked $n$ times to yield an $n\times K$ matrix) is feasible in $P'$ by definition, which proves that $\val{P}\ge\val{P'}$. The result follows.
\end{proof}
}

\leorev{We define the Eulerian cycle-based relaxation formulation, derived from Eq.~\eqref{dgprel} by adjoining Eq.~\eqref{xysys2} and Eq.~\eqref{centroid}, as follows:}
\begin{equation}
  \left.
  \begin{array}{rrcl}
    \min\limits_{s\ge 0, x,y} & \sum\limits_{\{i,j\}\in E} (s_{ij}^+ + s^-_{ij}) && \\
  \forall (i,j)\in A(\tilde{G}) &  \sum\limits_{k\le K} y_{ijk}^2 - d_{ij}^2 &=& s^+_{ij} - s^-_{ij} \\
  \forall k\le K & \sum\limits_{(i,j)\in \mathscr{C}} c^{\mathscr{C}}_{ij}y_{ijk} &=& 0 \\
  \forall (i,j)\in A(\tilde{G}) & x_{ik}-x_{jk} &=& y_{ijk} \\
  \forall k\le K & \sum\limits_{i\in V} x_{ik} &=& 0.
  \end{array}\right\}\label{dgprel1}
\end{equation}

\leorev{
\begin{proposition}
  Eq.~\eqref{dgprel1} is a relaxation of the EDGP.
\end{proposition}
\begin{proof}
  Let us call Eq.~\eqref{dgprel} $R$ and Eq.~\eqref{dgpnew} $P$. By Prop.~\ref{prop:dgprel}, $R$ is a relaxation of $P$. By adjoining new variables $x$ and Eq.~\eqref{xysys2} as constraints to both $R$ and $P$, we obtain formulations $R',P'$ such that $R'$ is a relaxation of $P'$. But by Lemma \ref{lem:xysys2} we have that $\val{P'}=\val{P}$, so $R'$ is a relaxation of $P$, which is a valid formulation of the EDGP. Note that Eq.~\eqref{dgprel1} is $R'$ with the centroid constraints Eq.~\eqref{centroid} adjoined. By Lemma \ref{lem:centroid}, therefore, $\val{R'}=\val{\ref{dgprel1}}$. Thus, Eq.~\eqref{dgprel1} is a relaxation of the EDGP.
\end{proof}
}

\leorev{
  \begin{remark}
    \label{rem:dgprel1}
    In general, we have $\val{\ref{dgprel1}}\ge \val{\ref{dgprel}}$, since Lemma \ref{lem:xysys2} only holds for Eq.~\eqref{dgpnew}, but not for Eq.~\eqref{dgprel}, as mentioned under Prop.~\ref{prop:dgprel}. Therefore Eq.~\eqref{dgprel1} is a tighter relaxation than Eq.~\eqref{dgprel}.
  \end{remark}
}

\section{Computational experiments}
\label{s:compres}
The aim of this section is to compare the computational performance of the following EDGP formulations:
\begin{enumerate}[(i)]
\item the cycle-based formulation in Eq.~\eqref{dgpnew}, where the realization is retrieved as a post-processing stage using \eqref{xysys2} according to Cor.~\ref{2stage};
\item the Eulerian cycle-based relaxation in Eq.~\eqref{dgprel1};
\item the classic edge-based formulation in Eq.~\eqref{go1}.
\end{enumerate}
All of these formulations are nonconvex Nonlinear Programs (NLP), which are generally $\mathsf{NP}$-hard to solve. More specifically, all of these formulations are as hard to solve as the EDGP, which is $\mathsf{NP}$-hard.

As a solution algorithm, we used a very simple MultiStart (MS) heuristic based on calling a local NLP solver from a random initial starting point at each iteration, and updating the best solution found so far as needed: although there are better heuristics around \cite{dvnsjogo,zoo,mwu}, MS is the best trade-off between implementation simplicity and efficiency. Moreover, more efficient heuristics often change the formulation during their execution, which may hinder the meaning of this computational comparison between formulations.

We evaluate the quality of a realization $x$ of a graph $G$ according to mean (MDE) and largest distance error (LDE), defined this way:
\begin{eqnarray*}
  \mathsf{mde}(x,G) &=& \frac{1}{|E|} \sum\limits_{\{i,j\}\in E} \big| \|x_i-x_j\|_2 - d_{ij}\big|\\
  \mathsf{lde}(x,G) &=& \max\limits_{\{i,j\}\in E} \big| \|x_i-x_j\|_2 - d_{ij}\big|.
\end{eqnarray*}

\leorev{Furthermore, for each realization $x$ of a graph $G$ found by using the MS algorithm, we consider the value of the corresponding solution $\mathsf{solVal}(x, G)$. We note that, due to the heuristic nature of the MS, this value is not guaranteed tobe globally optimal.}

The CPU time taken to find the solution may also be important, depending on the application. In the control of underwater vehicles \cite{bahr}, for example, DGP instances might need to be solved in real time. In other applications, such as finding protein structure from distance data \cite{bipbip,jcim} (our application of choice), the CPU time is not so important.

Our tests were carried out on a single CPU of a 2.1GHz 4-CPU 8-core-per-CPU machine with 64GB RAM running Linux. The local NLP solver used within the MS heuristic was the IPOpt solver \cite{ipopt}. We remarked in some preliminary tests that IPOpt was considerably slowed down by variants of Eq.~\eqref{go} such as Eq.~\eqref{minabs}, which essentially move a nonconvexity on the objective to one in the constraints. The same holds for the cycle-based formulation in Eq.~\eqref{dgpnew}. We therefore reformulated Eq.~\eqref{dgpnew} as follows:
\begin{equation}
  \left.
  \begin{array}{rrcl}
    \min\limits_{y} & \sum\limits_{\{i,j\}\in A(\bar{G})} (\sum\limits_{k\le K} y_{ijk}^2 - d_{ij}^2)^2 && \\
  \forall k\le K, B\in\mathcal{B} & \sum\limits_{(i,j)\in A(B)} c^B_{ij}y_{ijk} &=& 0,
  \end{array}\right\}\label{dgpnew1}
\end{equation}
and Eq.~\eqref{dgprel1} similarly.

Our implementation consists of a mixture of Python 3 \cite{python3} and AMPL \cite{ampl} interfaced through \texttt{amplpy}. Cycle bases and Eulerian cycles are found using \texttt{networkX} \cite{networkX}. Solutions to the feasible but possibly overdetermined linear systems in Eq.~\eqref{xysys2} are obtained using an $\ell_1$ error minimization approach reformulated as a Linear Programming problem solved with CPLEX \cite{cplex129}.

\subsection{Results}
A benchmark on a diverse collection of randomly generated weighted graphs of small size and many different types, with a very similar set-up to the one discussed here, is presented in \cite{cycledgp-ctw20}. It was found that the cycle formulation finds better MDE values, while the edge formulation generally finds better LDE values and is faster. Some results on proteins, obtained with only 3 MS iterations, were also presented in \cite{cycledgp-ctw20}.

The benchmark we consider here contains medium to large scale protein graph instances realized in $\mathbb{R}^3$, all of which contain cycles. W.r.t.~the protein results presented in \cite{cycledgp-ctw20}, we integrated one more instance, \texttt{1tii}, which, at 69800 edges and 5684 vertices, is considerably larger than all the others. The results are given in \leorev{Tables \ref{t2} and \ref{t_solvals}.

In Table~\ref{t2},} we report instance name, instance sizes $m$ and $n$, then performance measures MDE, LDE and CPU for cycle, Eulerian and edge-based formulations. In the last three lines we report average, standard deviation, and number of instances where the formulation performed best, for all performance measures. In all tested cases, finding the cycle basis, the Eulerian cycles, and solving Eq.~\eqref{xysys2} took a small fraction of the total solution time. The missing result for instance \texttt{100d} on the Eulerian cycle reformulation is due to a failure occurred in the \texttt{networkX} module because the graph of \texttt{100d} is not connected.
\begin{table}[!ht]
  {\scriptsize
\begin{center}
  \hspace*{-0.8cm}\begin{tabular}{lrr|rrr|rrr|rrr} \hline
&   &  & \multicolumn{3}{c|}{MDE} & \multicolumn{3}{c|}{LDE} & \multicolumn{3}{c}{CPU} \\
{\it Instance}     & $m$  & $n$ & cycle & Eul   & edge      & cycle & Eul   & edge      & cycle    & Eul   & edge           \\ \hline
\texttt{1guu}      & 955  & 150 & 0.086 & 0.069 & \B{0.053} & 1.234 & 1.068 & \B{1.037} & \B{7.90} & 553.76   & 290.21      \\
\texttt{1guu-1}    & 959  & 150 & 0.080 & 0.082 & \B{0.059} & 1.013 & 1.069 & \B{0.980} & 9.67	   & 23.03    & \B{1.72}    \\
\texttt{1guu-4000} & 968  & 150 & 0.112 & 0.106 & \B{0.092} & 1.073 & 1.431 & \B{0.936} & 8.68	   & 10.77    & \B{1.56}    \\
\texttt{pept}      & 999  & 107 & \B{0.144} & 0.239 & 0.179 & 2.862 & \B{1.847} & 1.943 & 5.52	   & 4.72     & \B{1.4}     \\
\texttt{2kxa}     & 2711  & 177 & \B{0.051} & 0.119 & 0.172 & 3.705 & \B{2.826} & 3.813 & 21.53	   & 25.54    & \B{7.35}    \\
\texttt{res\_2kxa}& 2627  & 177 & \B{0.055} & 0.237 & 0.156 & \B{2.949} & 3.570 & 3.054 & 20.84	   & 21.20    & \B{12.44}   \\
\texttt{C0030pkl} & 3247  & 198 & \B{0.000} & 0.145 & 0.211 & \B{0.000} & 3.537 & 3.829 & 29.50	   & 26.69    & \B{7.36}    \\
\texttt{cassioli} & 4871  & 281 & 0.146 & 0.113 & \B{0.057} & 3.914 & 3.616 & \B{3.185} & 47.23	   & 48.44    & \B{14.51}   \\
\texttt{100d}     & 5741  & 488 & \B{0.201} & -     & 0.251 & \B{3.038} & -     & 3.987 & 387.32   & -        & \B{29.42}   \\
\texttt{hlx\_amb} & 6265  &392  & \B{0.105} & 0.214 & 0.119 & 3.836 & 3.888 & \B{3.485} & 120.25   & 80.27    & \B{20.54}   \\
\texttt{water}   & 11939  & 648 & \B{0.146} & 0.490 & 0.243 & \B{3.579} & 4.196 & 4.281 & 1346.69  & 399.42   & \B{224.66}  \\
\texttt{3al1}    & 17417  & 678 & \B{0.062} & 0.126 & 0.216 & 3.451 & \B{3.175} & 4.059 & 835.10   & 433.69   & \B{123.45}  \\
\texttt{1hpv}    & 18512 & 1629 & \B{0.385} & 0.402 & 0.416 & 3.847 & \B{3.831} & 4.015 & 10138.00 & 2387.29  & \B{442.70}  \\
\texttt{il2}     & 45251 & 2084 & 0.385 & \B{0.049} & 0.107 & 4.422 & \B{4.204} & 4.583 & 18141.22 & 9904.81  & \B{5255.76} \\
\texttt{1tii}    & 69800 & 5684 & 0.620 & 0.436 & \B{0.434} & 6.755 & 4.492 & \B{3.854} & 18846.37 & 38230.21 & \B{9039.28} \\ \hline
\textit{avg}     &       &      & \B{0.172} & 0.202 & 0.184 & \B{3.045} & 3.054 & 3.136 & 3331.05  & 3724.99  & \B{1031.49} \\
\textit{stdev}   &       &      & 0.167 & 0.144 & \B{0.118} & 1.673 & \B{1.204} & 1.272 & 6672.49  & 10272.3  & \B{2587.33} \\
\textit{$|\mbox{best}|$} &   &  & \B{9} &  1    &    5      &    4  &    5      & \B{6} &  1       & 0        & \B{14}
  \end{tabular}
\end{center}
  }
\caption{Cycle formulation vs.~Eulerian relaxation \gabjulrev{ vs.~edge formulation} performance\gabjulrev{s} on protein graphs (realizations in $K=3$ dimensions).}
\label{t2}
\end{table}

It appears that, on average, there is relatively little difference between the quality performances of these three \gab{E}DGP formulations on protein graphs of medium and large sizes. CPU-time wise, of course, the edge formulation is best. Cycle formulations, taken together, outperform the edge formulation on quality measures. The cycle-based formulation Eq.~\eqref{dgpnew} is slightly better than the other formulations for both MDE and LDE. The number of instances on which Eq.~\eqref{dgpnew} is best on quality measures is 13, against 11 for the edge-based formulation.

\leorev{In Table~\ref{t_solvals}, we report instance name and $\mathsf{solVal}$ for the cycle and the edge EGDP formulations. The three lines at the bottom of the table show the arithmetic and geometric mean of each column (``arithmean'' and ``geomean''), and the percentage of instances where the solution values of each formulation are smaller than the other (``best'').}

\begin{table}[!ht]
{\scriptsize
\begin{center}
\leorev{\begin{tabular}{lrr}
\hline
 & \multicolumn{2}{c}{solval} \\
Instance & cycle & edge \\
\hline
\texttt{1guu} & 9.27E+02 & \textbf{4.73E+02} \\
\texttt{1guu-1} & 8.91E+02 & \textbf{5.67E+02} \\
\texttt{1guu-4000} & 1.40E+03 & \textbf{1.01E+03} \\
\texttt{pept} & \textbf{3.21E+03} & 3.50E+03 \\
\texttt{2kxa} & \textbf{3.04E+03} & 1.25E+04 \\
\texttt{res\_2kxa} & \textbf{3.42E+03} & 9.81E+03 \\
\texttt{C0030pkl} & \textbf{0} & 1.92E+04 \\
\texttt{cassioli} & 2.37E+04 & \textbf{7.73E+03} \\
\texttt{100d} & \textbf{3.16E+04} & 4.36E+04 \\
\texttt{hlx\_amb} & 2.04E+04 & \textbf{1.97E+04} \\
\texttt{water} & \textbf{5.73E+04} & 1.10E+05 \\
\texttt{3al1} & \textbf{2.56E+04} & 1.22E+05 \\
\texttt{1hpv} & \textbf{2.71E+05} & 3.03E+05 \\
\texttt{il2} & 7.76E+05 & \textbf{1.46E+05} \\
\texttt{1tii} & 2.33E+06 & \textbf{1.23E+06} \\
\hline
arithmean & 2.37E+05 & 1.36E+05 \\
geomean & 5.96E+02 & 1.86E+04 \\
best  & 53.33\% & 46.67\% \\
\end{tabular}}%
\end{center}
}
\caption{\leorev{MS solution values of cycle formulation vs.~edge formulation (realizations in $K=3$ dimensions)}}
\label{t_solvals}
\end{table}

\leorev{The cycle formulation reports better local optima more often than the edge formulation (``best'' = 53.3\%), while the latter is more stable, on average, as its arithmetic mean is slightly smaller. However, since the solution values of the cycle formulation are sometimes much smaller than those of the edge formulation, the geometric mean of the former is about two orders of magnitude smaller than that of the latter.}

We observe that Eq.~\eqref{dgpnew} was the only formulation by which a global optimum was found (that of \texttt{C0030pkl}) \leorev{using MS. Overall, the results reported in Table~\ref{t_solvals} follow those of Table~\ref{t2}, except in the case of instance \texttt{hlx\_amb}.}

\leorev{We decided to ignore the Eulerian formulation in Table \ref{t_solvals}, as its the objective function values were often larger than those of the corresponding cycle formulation, despite the fact that the former is a relaxation of the latter. This apparent anomaly is due to the heuristic nature of the MS solution algorithm.}

All in all, we believe that our results show that cycle formulations are credible competitors w.r.t.~the well established edge-based formulations, especially when the CPU time is not an important performance measure (which is generally the case in the protein conformation application).

\section*{Acknowledgements}
While the seminal idea for considering DGPs over cycles dates from Saxe's {\bf NP}-hardness proof \cite{saxe79}, the ``cycle formulation'' concept occurred to us as one of the authors (LL) attended a talk by Matteo Gallet given at the Erwin Schr\"odinger Institute (ESI), Vienna, during the Geometric Rigidity workshop 2018. LL has received funding from the European Union's Horizon 2020 research and innovation programme under the Marie Sklodowska-Curie grant agreement n.~764759 ``MINOA'', and from the ANR PRCI project ``MultiBioStruct''. CL is grateful to the Brazilian research agencies FAPESP and CNPq for support. NM is grateful to the Brazilian research agencies COPPETEC Foundation and CNPq for support.

\bibliographystyle{plain+eid}
\bibliography{cycledgp2}

\clearpage
\appendix
\onecolumn

\end{document}